\documentclass[a4paper,11pt]{article}
\usepackage[pagewise]{lineno}
\usepackage{amsmath}
\usepackage{amssymb}
\usepackage{mathrsfs}
\usepackage{amsfonts}
\usepackage{amsthm}
\usepackage{pstricks, pst-node, pst-text, pst-3d,psfrag}
\usepackage{graphicx}
\usepackage{indentfirst}
\usepackage{enumerate}
\usepackage{subfigure}
\usepackage{cite}
\usepackage[colorlinks=true]{hyperref}
\hypersetup{urlcolor=blue, citecolor=red}
\usepackage{soul}
\usepackage{footmisc}

\theoremstyle{plain}
\newtheorem{thm}{Theorem}[section]
\newtheorem{prop}[thm]{Proposition}
\newtheorem{cor}[thm]{Corollary}
\newtheorem{lem}[thm]{Lemma}
\theoremstyle{definition}

\theoremstyle{remark}

\numberwithin{equation}{section}

\makeatletter 
\@addtoreset{equation}{section}
\makeatother  
\renewcommand\theequation{\oldstylenums{\thesection}%
                   .\oldstylenums{\arabic{equation}}}

\newcommand\abs[1]{\lvert#1\rvert}

\newcommand\RR{\ensuremath{\mathbb{R}}}

\begin{document}

\title{Birkhoff Center and Statistical Behavior of Competitive Dynamical Systems\thanks{Supported by the National Natural Science Foundation of China (NSFC 11825106 and 12331006).}}

\setlength{\baselineskip}{16pt}

\author{
Xi Sheng and Yi Wang
\\[2mm]
School of Mathematical Sciences\\
University of Science and Technology of China\\
Hefei, Anhui, 230026, P. R. China
\\[2mm]
Yufeng Zhang\footnote{Partially supported by the National Key R\&D Program of China (2022YFA1005801), the China Postdoctoral Science Foundation (2023TQ0233) and the Jiangsu Outstanding Postdoctoral Funding Program.} \thanks{Corresponding author: zyfp@mail.ustc.edu.cn (Y. Zhang).}
\\[2mm]
School of Mathematical Sciences\\ 
Center for Dynamical Systems and Differential Equations\\
Soochow University\\ 
Suzhou, Jiangsu, 215006, P. R. China
}

\date{}
\maketitle

\begin{abstract}
We investigate the location and structure of the Birkhoff center for competitive dynamical systems, and give a comprehensive description of recurrence and statistical behavior of orbits. An order-structure dichotomy is established for any connected component of the Birkhoff center, that is, either it is unordered, or it consists of strongly ordered equilibria. Moreover, there is a canonically defined countable disjoint family $\mathcal{F}$ of invariant $(n-1)$-cells such that each unordered connected component of the Birkhoff center lies on one of these cells. We further show that any connected component of the supports of invariant measures either consists of strongly ordered equilibria, or lies on one element of $\mathcal{F}$.
In particular, any $3$-dimensional competitive flow has topological entropy $0$.
\end{abstract}

\section{Introduction}

In the biological science and population ecology, competition means that increasing any one of species can only have a negative effect on the per capita growth rate of any other species. A mathematical model of competition between species of organisms frequently involves a system of differential equations \cite{F80,H91,M73}, for which all the off-diagonal entries of its linearized Jacobian matrix are nonpositive. Contrary to popular belief before 1980's, the construction of Smale \cite{S76} showed that an arbitrary smooth flow in the simplex spanned by the unit coordinate vectors can be embedded as an attractor in a competitive system. This reveals that mathematical models of competition in higher dimensions can lead to differential equations with complicated dynamics.

A beautiful geometric theory of the dynamical behavior of competitive systems was initiated by a series of famous works of Hirsch \cite{H82-1,H85-2,H88-3,H88-S}, who showed that competitive systems have a special overall structure regarding individual orbits. For instance, the $\omega$-limit set of any
orbit is unordered (see \cite[Theorem 2.2]{H82-1} or \cite[Theorem 3.3.2]{H95}). Moreover, under mild additional restrictions, there is a canonically defined countable family of unordered, disjoint Lipschitz open $1$-codimensional cells such that the $\omega$-limit set of every non-convergent orbit lies on one of the cells (see \cite{H88-3}). These indicate that the dynamics of competitive systems can be no worse than that of completely general systems of one less dimension. It also means that Smale's construction is not as special as it seems.
Over three decades since its development, the theories of competitive systems and applications have undergone
extensive investigations and continue to expand. One may refer to the monographs and reviews (\cite{HS05,LL-19,M-20,H95,Smi99,Smi17} and references therein) for more details.

From the point of view of abstract dynamic complexity of individual orbits, Hirsch's theory, together with Smale's construction, has been interpreted as meaning that the competitive system is essentially $1$-codimensional.
But from a recurrence and statistical behavior of orbits perspective, it is not yet able to provide feasible and effective information. As a consequence, it pays to be on the lookout for the recurrence and statistical behavior of orbits for competitive systems.

In the terminology of dynamical systems, the presence of non-periodic recurrence turns out to be the first indication of complicated asymptotic behavior. The closure of the set of all recurrent points, called {\it Birkhoff center}, is the hub of recurrence behavior (c.f. \cite{K95,M12}). Statistical behavior of orbits, namely, typical recurrent behavior with respect to an invariant measure (in the measure-theoretic sense)
can be viewed as  stronger ``quantitative" counterparts of ``qualitative" recurrence properties. In particular, the supports of invariant measures are contained in the Birkhoff center (see, e.g. \cite{M12,J20}).

The purpose of the present paper is devoted to the investigation of the location and structure of the Birkhoff center, and gives a comprehensive description of the recurrence and statistical behavior of orbits for competitive systems.
Following Hirsch \cite{H88-3} (or \cite{HS05}), the flow $\Phi_t$ generated by a competitive system of differential equations
preserves a partial order ``$\le$" (induced by a convex cone $C^+$ with nonempty interior {\rm Int}$C^+$) in backward time, that is, $\Phi_t(x)\le \Phi_t(y)$ whenever $x\le y$ and $t<0$. For brevity, we hereafter call $\Phi_t$ a competitive flow. Throughout the paper, we write $x\leq y$ if $y-x\in C^+$. A subset $Y\subset \mathbb{R}^{n}$ is called {\it unordered} if none of its points are related by ``$\leq$". $Y$ is {\it totally ordered} if any of its points are related by ``$\leq$".

The main results of this paper are the following universal location description for the Birkhoff center $\mathcal{B}(\Phi)$ of $\Phi_t$. Under the standard assumptions of dissipation and irreducibility, we first establish an {\it order-structure dichotomy} for any connected component $B$ of the Birkhoff center, that is, $B$ is either unordered; or a totally ordered set of equilibria (see Theorem \ref{Thm_2.1} or Theorem \ref{4.1}). Moreover, we show that there is a canonically defined countable disjoint family $\mathcal{F}$ of invariant open $(n-1)$-cells such that each unordered connected component of the Birkhoff center lies on one of these cells (see Theorem \ref{Thm_2.1}). Compared with those cells constructed in Hirsch \cite{H88-3}, new vital elements are introduced into the family $\mathcal{F}$ in our framework. These cells have nice geometrical properties and are determined by the system in a way to be described later (see Theorem \ref{4.2}).

A straightforward consequence of our main results is that any connected component of the supports of invariant measures is either unordered; or a totally ordered set of equilibria (see Theorem \ref{Thm_2.2}), which is one of the most exciting fallout of Theorem \ref{Thm_2.1}. It reveals that the canonically defined countable disjoint family $\mathcal{F}$ of invariant open $(n-1)$-cells, is actually covering all the unordered connected components of the Birkhoff center $\mathcal{B}(\Phi)$. Consequently, from this global and geometrical perspective, the recurrence and statistical behavior of orbits for competitive systems have a special overall structure.
In particular, we show that any competitive flow in $\mathbb{R}^{3}$ has topological entropy $0$ (see Corollary \ref{Cor_2.3}).
In our forthcoming works \cite{JSWZ-LSC-competitive,SWZ-SAA-1}, the above main results will be further applied to the asymptotic behavior of stochastic  approximation, as well as stochastic perturbations and stability of differential equations.

It deserves to point out that the structure of the Birkhoff center and the supports of invariant measures indicate the (typical) recurrence and statistical behavior, which cannot solely rely on the analysis of the limit sets of individual orbits. Accordingly, in order to show the order-dichotomy for the connected component $B$ of the Birkhoff center, one of the innovations in current work is to introduce a so-called ``joint cone-boundary" idea to locate the the Birkhoff center. As a matter of fact, our approach is heavily based on such new viewpoint and the techniques developed around it.

To be more precise, let $C=C^+\cup (-C^+)$, and let $K$ be the closure of the complement of $C$. Clearly, $C$ is a cone of rank $1$ (abbr. $1$-cone), and $K$ is the complementary cone of $C$ (see, e.g. \cite{FWW17,FWW19,Z-22,Margaliot-21} and references therein). For each $x\in \RR^n$, we define a translated cone pair $(C_x,K_x)$ at $x$, where $C_x=x+C$ and $K_x=x+K$ are the cones translated to the base point $x$. The boundary $\partial C_x$, which equals $\partial K_x$, is called the  {\it joint cone-boundary} of $(C_x,K_x)$. The kernel idea of our approach is to show an \textit{intersection principle for joint cone-boundary}: For any recurrent point $x$, the joint cone-boundary $\partial C_x$ intersects with the Birkhoff center $\mathcal{B}(\Phi)$ only at $x$ itself (see Proposition \ref{endpoint}). Together with the connectedness of the component $B$, such critical insight will enable us to obtain the order dichotomy in Theorem \ref{Thm_2.1} (or Theorem \ref{4.1}), that is, $B$ is either unordered; or a totally ordered set of equilibria.

Meanwhile, we have to emphasize that the proof of this intersection principle is highly nontrivial, whose proof can roughly break into the two steps. The first step is to show that for any recurrent point $x$, the intersection of $\partial C_x$ with the recurrence set $R(\Phi)$ is just $x$ itself (see Proposition \ref{sep_n}). This can be viewed as a ``soft" version of Proposition \ref{endpoint}, since $R(\Phi)$ is contained in the Birkhoff center $\mathcal{B}(\Phi)$ (In fact, $\mathcal{B}(\Phi)$ is the closure of $R(\Phi)$). We manage to accomplish it by analyzing the properties of the IP recurrent-time set of $x$ for the competitive flow $\Phi_t$ (see Proposition {\red A.1} in the Appendix, which is motivated by Furstenberg \cite[Proposition 8.10]{F81} for discrete-time systems). The second step is to reinforce Proposition \ref{sep_n} to the ``solid" version Proposition \ref{endpoint}. For this purpose, one needs to generalize the intersection principle for joint cone-boundary for the subsets with weaker recurrence, for instance, the limit set of recurrent points. To overcome such difficulty, we established a limit-set dichotomy in the category of recurrent points (see Lemma \ref{limit_set_dichotomy}), which helps us accomplish our approach.

Meanwhile, we also point out that the construction of the family $\mathcal{F}$ of disjoint invariant open $(n-1)$-cells in Theorem \ref{Thm_2.1} (see also Theorem \ref{4.2}), is inspired by Hirsch \cite[Theorem 4.3]{H88-3}.
However, compared with Hirsch's construction for Kolmogorov competitive systems (which are only restricted in the first quadrant of $\mathbb{R}^n$), more new elements need to be collected into the family $\mathcal{F}$ in our framework. 
Furthermore, we note that Hirsch's key lemma \cite[Lemma 3.3]{H88-3}, called the absorbing principle for the $\omega$-limit sets, plays a critical role for proving that the $\omega$-limit set lies on one element of $\mathcal{F}$.
In our situation, a difficult question naturally arises that whether such absorbing principle can still be valid for the component $B$ of $\mathcal{B}(\Phi)$, instead for the $\omega$-limit sets, because the analysis of $\mathcal{B}(\Phi)$ cannot be achieved only by the information of the limit-sets of individual orbits.
Fortunately, together with the connectedness of the component $B$, the intersection principle for joint cone-boundary (Proposition~\ref{endpoint}) and connecting lemma (Lemma \ref{3.6_2}, which is an important consequence of Proposition \ref{endpoint}), help us overcome such difficulty and ensure the feasibility of the absorbing principle for the Birkhoff center (see Lemma \ref{k_2}), by which we succeed in obtaining our main results.

This paper is organized as follows. In Section \ref{s2}, we introduce some notations, provide important definitions, and state our main results. In Section \ref{s3}, we explore and present the intersection principle for joint cone-boundary, which plays a key role in the proof of our main theorems. Finally, in Section \ref{Limit dich}, we characterize the geometric structure of the Birkhoff center by defining the family $\mathcal{F}$ of disjoint invariant open $(n-1)$-cells and complete the proof of the main theorems. 
In the appendix, we provide the critical properties of recurrent-time sets which are important for proving the intersection principle.

\noindent\section{Notations and Main Results}\label{s2}
Let $(\mathbb{R}^{n},\left|\cdot \right|)$ be the $n$-dimensional Euclidean space.
A closed subset $C^{+}\subset\mathbb{R}^{n}$ is a \textit{convex cone} if it satisfies: $C^{+} +C^{+} \subset C^{+}$, $\alpha\cdot C^{+}\subset C^{+}$ for all $\alpha > 0$, and $C^{+}\cap(-C^{+}) = \{0\}$.
$C^+$ is called \textit{solid} if its interior $\text{Int}C^+\neq\emptyset$.
A cone $C^+$ naturally induces a (partial) order ``$\leq$" by $x\leq y$ whenever $y-x\in C^+$.
We write $x<y$ if $y-x\in C^+\backslash\{0\}$ and $x\ll y$ if $y-x\in \text{Int}C^{+}$;
A subset $S\subset\mathbb{R}^n$ is \textit{unordered} if $S$ cannot contain distinct points related by $``\leq"$.
$S$ is called \textit{ordered} (resp. \textit{strongly ordered}) if any two points in $S$ are related by $``\leq"$ (resp. $``\ll"$).
Given $x,y\in\mathbb{R}^n$, the set $[[x, y]] = \{z\in\mathbb{R}^n:x\ll z\ll y\}$ is called an \textit{open order interval}, and we write \textit{closed order interval} $[x, y] =\{z\in\mathbb{R}^n:x\leq z\leq y\}$. In
particular, we write $[[x,+\infty]]=\{y\in\mathbb{R}^n: x\ll y\}$ and $[[-\infty,x]]=\{y\in\mathbb{R}^n: y\ll x\}$.

A \textit{flow} $\Phi=\{\Phi_t\}_{t\in\mathbb{R}}$ on $\mathbb{R}^{n}$ is a continuous map $\Phi:\mathbb{R}\times\mathbb{R}^{n}\rightarrow\mathbb{R}^{n}$ with $\Phi_{0}=\text{Id}$ and $\Phi_{t}\circ\Phi_{s}=\Phi_{t+s}$ for $t,s\in\mathbb{R}$. Here, $\Phi_{t}(\cdot)=\Phi(t,\cdot)$ for $t\in\mathbb{R}$ and \text{Id} is the identity map on $\mathbb{R}^{n}$.
A \textit{positive orbit} of $x$ is denoted by $O^{+}(x) = \left\{\Phi_{t}(x):t\geq0\right\}$; a \textit{negative orbit} of $x$ is denoted by $O^{-}(x) = \left\{\Phi_{t}(x):t\leq0\right\}$;
and an \textit{orbit} of $x$ is denoted by $O(x) = \left\{\Phi_{t}(x):t\in\mathbb{R}\right\}$.
An \textit{equilibrium} $p\in\mathbb{R}^{n}$ is a point for which $O(p)=\left\{p\right\}$. The set of all equilibria of $\Phi_t$ is denoted by $E$.
The \textit{$\omega$-limit set} $\omega(x)$ of $x$ is defined by $\omega(x)=\bigcap_{t\geq0}\overline{\bigcup_{\tau\geq t}\Phi_{\tau}(x)}$, the \textit{$\alpha$-limit set} $\alpha(x)$ of $x$ is defined by $\alpha(x)=\bigcap_{t\leq0}\overline{\bigcup_{\tau\leq t}\Phi_{\tau}(x)}$, where the closure of a set $S\subset\mathbb{R}^n$ is denoted by $\bar{S}$.
A subset $D\subset\mathbb{R}^{n}$ is called \textit{invariant} if $\Phi_{t}D= D$ for all $t\in\mathbb{R}$.
A set $\Gamma\subset\mathbb{R}^n$ is an \textit{attractor} (or a \textit{repeller}) for $\Phi_t$ provided $\Gamma$ is a nonempty compact invariant set that has a \textit{neighborhood} $N$ of $\Gamma$ such that $\lim_{t\rightarrow\infty}d(\Phi_{t}(x),\Gamma)=0$  (or $\lim_{t\rightarrow\infty}d(\Phi_{-t}(x),\Gamma)=0$) uniformly in $x\in N$, where $d(a,\Gamma)=\inf_{b\in \Gamma}|a-b|$.

A point $x\in\mathbb{R}^{n}$ is called a \textit{recurrent point} if there exists a sequence $t_{i}\to\infty$ such that $\Phi_{t_{i}}(x)\rightarrow x$ as $t_{i}\rightarrow \infty$.
Denote by $R(\Phi)$ all recurrent points of $\Phi_t$ in $\mathbb{R}^{n}$.
The closure of $R(\Phi)$ in $\mathbb{R}^{n}$ is called {\it Birkhoff center}, labeled by $\mathcal{B}(\Phi)$, i.e., $$\mathcal{B}(\Phi)=\overline{\left\{x\in\mathbb{R}^{n}:x\in\omega(x)\right\}}.$$

The flow $\Phi_t$ is called \textit{competitive} if $x\leq y$ whenever $\Phi_{t}(x)\leq \Phi_{t}(y)$ for some $t>0$; and called \textit{strongly competitive} if $x\ll y$ whenever $\Phi_{t}(x)< \Phi_{t}(y)$ for some $t>0$. 
As one may know, in a system of differential equations, the strongly competitive property can be realized by the irreducibility of the linearized Jacobian matrices (see \cite{H95,HS05}).

\vskip 2mm
Throughout the paper, we assume the following standing assumptions:
\vskip 2mm

\noindent {\bf (H1)} (Competition) The flow $\Phi_{t}$ is strongly competitive in $\mathbb{R}^n$; \vskip 1mm

\noindent {\bf (H2)} (Dissipation) There is a compact attractor $\Gamma$ which uniformly attracts each compact set 

\quad\ of the initial values.
\vskip 2mm

The following theorems are the main results of our paper.
\begin{thm}\label{Thm_2.1}
Assume that \textnormal{(H1)-(H2)} hold. Let $B$ be a connected component of Birkhoff center $\mathcal{B}(\Phi)$. Then one of the following alternatives must occur:\par
{\rm(i)} $B$ is unordered; or otherwise,\par
{\rm(ii)} $B$ consists of strongly ordered equilibria.\\
Moreover, there is a canonically defined countable disjoint family $\mathcal{F}=\{M_i\}$ of invariant open $(n-1)$-cells such that any unordered connected component $B$ lies on one element of $\mathcal{F}$.
\end{thm}

Theorem \ref{Thm_2.1}, with slight stronger results (see Theorems \ref{4.1} and \ref{4.2}), will be proved in Section \ref{Limit dich}. More detailed information of the invariant $(n-1)$-cells in the family $\mathcal{F}$ will be presented in Section \ref{Limit dich}.

A straightforward consequence of Theorem \ref{Thm_2.1} is the description of the order structure of the supports of the invariant measure on $\mathbb{R}^n$.
More precisely, let $(\mathbb{R}^n, \mathcal{B}, \mu)$ be the probability space, where $\mathcal{B}$ is the $\sigma$-algebra of Borel sets in $\mathbb{R}^n$. Let $\mu$ be a $\Phi$-invariant probability measure, i.e., $\mu\circ\Phi_{-t}=\mu$ for all $t\geq0$, 
The support of $\mu$, denoted by ${\rm supp}(\mu)$, is the complement in $\mathbb{R}^n$ of the union of all open sets $U$ such that $\mu(U) = 0$. 
Since ${\rm supp}(\mu)\subset \mathcal{B}(\Phi)$ (see, e.g. \cite[p.28]{M12} or \cite[Appendix A]{J20}), Theorem \ref{Thm_2.1} directly implies the following

\begin{thm}\label{Thm_2.2}
Assume that \textnormal{(H1)-(H2)} hold. Let $B\subset {\rm supp}(\mu)$ be any connected component of the support of invariant measure $\mu$. Then either $B$ consists of strongly ordered equilibria; or $B$ lies on one element of $\mathcal{F}$.
\end{thm}

\vskip 2mm
\noindent\textbf{Remark 2.1.} 
The celebrated work by Hirsch \cite{H88-3} on competitive systems, which is mainly from the point of view of the complexity of the $\omega$-limit sets for individual orbits, has been interpreted in many literatures that the competitive system is essentially 1-codimensional. However, it deserves to point out that sometime such perspective is somewhat deceptive, because the study of global dynamics cannot boil down to understanding the structure of $\omega$-limit sets. In light of Theorems \ref{Thm_2.1} and \ref{Thm_2.2}, we have presented feasible and effective information for the recurrence and statistical behavior of orbits for competitive systems, which provide a significant supplement which guarantees that the competitive system is essentially 1-codimensional. 
\vskip 2mm

By virtue of Theorem \ref{Thm_2.2}, one can immediately obtain the zero-entropy characterization for 3-dimensional competitive flows.

\begin{cor}\label{Cor_2.3}
Any strongly competitive flow in $\mathbb{R}^3$ has topological entropy 0.
\end{cor}
\begin{proof}
Let $\mu$ be any ergodic invariant measure and $B={\rm supp}(\mu)$. Then $B$ is an invariant connected component of the Birkhoff center (see more details in Hirsch \cite[Proposition 1.8]{H99}) for the competitive flow $\Phi_t$. Hence, by Theorem \ref{Thm_2.2}, $B$ is either unordered, or consists of strongly ordered equilibria.
When $B\subset E$ is strongly ordered, it is clear that $B$ is just an equilibrium, because $\mu$ is ergodic. 
When $B$ is unordered, the competitive flow on $B\subset\mathbb{R}^3$ is topologically equivalent to a flow on a compact invariant set of systems in $\mathbb{R}^2$ (see Smith \cite[Theorem 3.3.4]{H95}).
Since any continuous flow on a compact 2-dimensional manifold has measure-theoretic entropy 0 (see Young \cite{Y77}). 
Therefore, in both case, $\Phi_t$ has zero-measure theoretic entropy for $\mu$.
By the Variational Principle (more details see \cite{D70}), one obtains that $\Phi_t$ has topological entropy 0.
\end{proof}
In our future work \cite{JSWZ-LSC-competitive,SWZ-SAA-1}, we will apply Theorem \ref{Thm_2.1} and Theorem \ref{Thm_2.2} to investigate the asymptotic behavior of stochastic approximation (c.f. \cite{B99}) or stochastic perturbations and stability of differential equations (c.f. \cite{J20,J23}).

\noindent\section{Intersection Principle for Joint Cone-Boundary}\label{s3}

In this section, we will establish the intersection principle for joint cone-boundary, which turns out to be a crucial tool to locate the Birkhoff center in the next section.

In order to explain this intersection principle precisely, we let $C=C^+\cup (-C^+)$ and let $K$ be the closure of $\mathbb{R}^n\backslash C$. $C$ is called a cone of rank 1, and $K$ is the complementary cone of $C$ (see \cite{FWW17,FWW19,Z-22,Margaliot-21} and references therein).
Fix each $x\in\mathbb{R}^n$, we define a translated cone pair $(C_x,K_x)$ at $x$, where $C_x=x+C$ and $K_x=x+K$. The boundary $\partial C_x$ (which equals to $\partial K_x$) is called the joint cone-boundary of $(C_x,K_x)$. 
Then, it is not difficult to see that $\Phi_t$ is competitive if and only if $\Phi_{-t}C_x\subset C_{\Phi_{-t}(x)}$ for any $t\geq0$, or equivalently, $\Phi_{t}K_x\subset K_{\Phi_{t}(x)}$ for any $t\geq 0$. In addition, it is strongly competitive if and only if $\Phi_{-t}C_x\backslash\{\Phi_{-t}(x)\}\subset {\rm Int}C_{\Phi_{-t}(x)}$ for any $t>0$, or equivalently, $\Phi_{t}K_x\backslash\{\Phi_{t}(x)\}\subset {\rm Int}K_{\Phi_{t}(x)}$ for any $t> 0$.

Roughly speaking, the intersection principle for joint cone-boundary (see Proposition \ref{endpoint}) says that, for any recurrent point $x$, $\partial C_x$ intersects with the Birkhoff center $\mathcal{B}(\Phi)$ only at $x$ itself.

Before proceeding our approach, we need more notations and preliminaries.
For any two closed subsets $A,B\subset\mathbb{R}^n$, the \textit{Hausdorff distance} between $A$ and $B$ is defined as $$d_{H}(A,B)=\max\{\sup_{a\in A} d(a,B),\sup_{b\in B} d(b,A)\},$$
where $d(a,B)=\inf_{b\in B}|a-b|$ and $d(b,A)=\inf_{a\in A}|a-b|$.
While, the \textit{separation index} between $A$ and $B$ is defined as $${\rm \underline{dist}}(A,B)=\inf_{x\in A,y\in B}|x-y|.$$
Clearly, ${\rm\underline{dist}}(A,B)>0$ if and only if $A\cap B=\emptyset$ when $A,B$ are compact. In particular, $d(x,B)={\rm \underline{dist}}(\{x\},B)$.
Label $A_x\triangleq x+A$ for any $A\subset \mathbb{R}^n$ and $x\in\mathbb{R}^n$.
Some useful relationships between $d_H$ and ${\rm\underline{dist}}$ are listed below (see \cite[Proposition 4.37]{R-09}):
\begin{align}
&d_{H}(A_x,A_y)\leq |x-y|;\label{dis_1}\\
&{\rm\underline{dist}}(A,B)\leq{\rm\underline{dist}}(A,C)+d_H(B,C),\label{dis_2}
\end{align}
where $A,B,C$ are closed sets in $\mathbb{R}^n$. In addition, we have 
\begin{lem}\label{4p}
For compact sets $A,B$ in $\mathbb{R}^n$,\par 
{\rm(i)} $d_{H}(A,\Phi_{t} A)$ is continuous with respect to $t$;\par
{\rm(ii)} Fix $t>0$, ${\rm\underline{dist}}(\Phi_{t}(x)-\Phi_{t}(A_x), B)$ is continuous with respect to $x$.
\end{lem}

\begin{proof}
(i). Clearly, $\sup_{u\in A}d(u,\Phi_{t}A)\leq \sup_{u\in A}|u-\Phi_{t}(u)|$. Given any compact set $A$, one can deduce that $\sup_{u\in A}|u-\Phi_{t}(u)|$ is continuous with respect to $t$, and hence, $\sup_{u\in A}d(u,\Phi_{t}A)$ is also continuous with respect to $t$. Similarly, one can obtain that $\sup_{v\in\Phi_{t}A}d(A,v)$ is continuous with respect to $t$. Therefore, $d_{H}(A,\Phi_{t}A)$ is continuous with respect to $t$.

(ii). Fix $t>0$. For any $x,y\in\mathbb{R}^{n}$, by (\ref{dis_2}), we have
\begin{equation*}
	\begin{split}
		|{\rm\underline{dist}}(\Phi_{t}(x)-\Phi_{t}(A_x), B)-& {\rm\underline{dist}}(\Phi_{t}(y)-\Phi_{t}(A_y), B)| \\
		&\leq d_{H}\left(\Phi_{t}(x)-\Phi_{t}(A_x), \Phi_{t}(y)-\Phi_{t}(A_y)\right) \\
		&\leq d_{H}\left(\Phi_{t}(x)-\Phi_{t}(A_x), \Phi_{t}(y)-\Phi_{t}(A_y)\right) + 
		d_{H}(\Phi_{t}(A_x), \Phi_{t}(A_y)).
	\end{split}
\end{equation*}
Note also that
\begin{align*}
	d_{H}(\Phi_{t}(A_x), \Phi_{t}(A_y)) 
	\leq d_{H}(\Phi_{t}(A_x), A_x) + d_{H}(A_x,A_y) + d_{H}(A_y, \Phi_{t}(A_y)).
\end{align*}
Then, together with (\ref{dis_1}) and Lemma \ref{4p}(i), we obtain that ${\rm\underline{dist}}(\Phi_t(x)-\Phi_{t}(A_x), B)$ is continuous with respect to $x$, which completes the proof.
\end{proof}

Now, we first present a ``soft" version of the intersection principle, that is,

\begin{prop}\label{sep_n}
	Let $x\in R(\Phi)$. Then $R(\Phi)\cap \partial C_x=\{x\}$.
\end{prop}
\begin{proof}	
Suppose that there exists $z\in R(\Phi)\cap\partial C_{x}$ with $z\neq x$. Write $W_{y}=\omega(z)\cap K_{y}$ for each $y\in\mathbb{R}^n$. Then, due to the invariance of $\omega(z)$ and the fact that $\Phi_t K_x\subset  K_{\Phi_t(x)}$ for $t\ge 0$, we have
\begin{equation}\label{E:in-W}
	\Phi_t W_x\subset W_{\Phi_t(x)}\subset \omega(z), \quad\textnormal{ for all } t\ge 0.
\end{equation}
We now {\it claim that there are some $\eta>0$ and $t_0>1$ such that} $${\rm \underline{dist}}(\Phi_{t_{0}}W_{x},\partial C_{x})>\eta\ \ and\ \ \Phi_{t_{0}}W_{x}\subset W_{x}.$$

Before we prove the claim, we first show how it deduces a contradiction, which helps us prove this proposition. In fact, by (\ref{dis_2}), one has  
\begin{equation}\label{d_d_h}
	{\rm \underline{dist}}(\Phi_{t_{0}+t} W_{x}, \partial C_{x})\geq {\rm \underline{dist}}(\Phi_{t_{0}} W_{x}, \partial C_{x})-d_{H}(\Phi_{t_{0}} W_{x}, \Phi_{t_{0}+t} W_{x}),\quad\text{ for all }t\in\mathbb{R}.
\end{equation}
For the last term on the right side of (\ref{d_d_h}), Lemma \ref{4p}(i) entails that there exists $\varepsilon_{0}>0$ such that
\begin{equation}\label{dis_betw_Wx_Wxt0}
	d_{H}\left(\Phi_{t_{0}} W_{x}, \Phi_{t_{0}+t} W_{x}\right)<\frac{1}{2}\eta,\quad \text{ for any } |t|<\varepsilon_{0}.
\end{equation}
Together with (\ref{dis_betw_Wx_Wxt0}), the first part of claim implies that (\ref{d_d_h}) can deduce that		
\begin{equation}\label{dis_betw_Wx_pCx}
	{\rm \underline{dist}}(\Phi_{t_{0}+t} W_{x}, \partial C_{x}) >\frac{1}{2}\eta,\quad \text{ for any } |t|<\varepsilon_{0}.	
\end{equation}
Denote the time-set by $\mathcal{T}(t_0,\varepsilon_0)\triangleq\{nt_0+t:n\in\mathbb{N},|t|<\varepsilon_0\}$. 
By the second part of claim $\Phi_{t_{0}}W_{x}\subset W_{x}$, we can reinforce (\ref{dis_betw_Wx_pCx}) into 
\begin{equation}\label{ds}
	{\rm \underline{dist}}(\Phi_{s} W_{x}, \partial C_{x})>\frac{1}{2}\eta,\quad \text{ for any } s\in\mathcal{T}(t_0,\varepsilon_0).
\end{equation}
Recall that $z\in W_x\cap\partial C_x$ (since $z\in R(\Phi)\cap\partial C_{x}$). Then, by (\ref{ds}), one obtains that $$|\Phi_{s}(z)-z|>\frac{1}{2}\eta,\quad \text{ for any } s\in \mathcal{T}(t_0,\varepsilon_0).$$
Hence, $$\{t>0: |\Phi_{t}(z)-z|<\frac{1}{2}\eta\}\cap \mathcal{T}(t_{0},\varepsilon_{0})=\emptyset.$$ This contradicts to the properties of the IP-set (Proposition {\red A.1}) in the Appendix. Thus, we have obtained $R(\Phi)\cap \partial C_{x}=\{x\}$.
	
Therefore, it remains to prove the claim. For this purpose, for each $t\ge 0$, we denote the set
$$\ W_{x}^{t}\triangleq x-\Phi_{t}(x)+\Phi_{t}W_{x}.$$ It then follows from \eqref{E:in-W} that
$W_x^t\subset x-\Phi_t(x)+W_{\Phi_t(x)}$. Recall that $K_x=x-\Phi_t(x)+K_{\Phi_t(x)}$. Then, 
\begin{equation}\label{E:wtx}
	W_x^t\subset K_x, \quad\textnormal{ for all } t\ge 0.
\end{equation}
Note that $x\notin\omega(z)$ (otherwise, $\omega(z)$ will possess two order-related points $x$ and $z$, contradicting the non-ordering of $\omega(z)$ (see \cite[Theorem 3.3.2]{H95})). Then, 	
one can find $\eta_2>\eta_1>0$ and a closed annulus domain $A=\{v\in \RR^n:\eta_1\le \abs{v-x}\le \eta_2\}$
such that $\omega(z)\subset {\rm Int}A$; and hence,
\begin{equation}\label{E:ome-z}
	{\rm \underline{dist}}(\omega(z),\,\partial C_{x}\backslash\text{Int}A)>0.\
\end{equation}
Consider the boundary $\partial C_{x}$ restricted to $A$, i.e., $\partial C_{x}^{A}=\partial C_{x}\cap A.$	We now {\it assert} that there is $\delta>0$ such that
\begin{equation}\label{E:wtx-bdd}
	{\rm \underline{dist}}\left(W_{x}^{t},\partial C_{x}^{A}\right)\geq\delta, \quad\textnormal{ for all } t> 1.
\end{equation}
In fact, we take $r>0$ large enough so that $\omega(z)\subset B_{r}(y)$ for all $y\in\overline{O^{+}(x)}$.
Define the compact set $R_{y}=\{v\in K_y: |v-y|\leq r\}$ for each $y\in \overline{O^{+}(x)}$, and it is clear that
\begin{equation}\label{W_R_K}
	W_{x}\subset R_{x}\subset K_x.
\end{equation}
Let $R_x^t=x-\Phi_t(x)+\Phi_1 R_{\Phi_{t-1}\left(x\right)}$.
Together with (\ref{E:in-W}), it follows from (\ref{W_R_K}) that
$\Phi_{t} W_x\subset\Phi_1 W_{\Phi_{t-1}(x)}\subset\Phi_1 R_{\Phi_{t-1}(x)}$ for any $t>1$. So, 
\begin{equation*}
	W_x^t\subset R_x^t,\quad \text{ for any } t>1.
\end{equation*}
Hence, we have
\begin{equation}\label{d>d}
	{\rm \underline{dist}}(W_{x}^{t},\partial C_{x}^{A})\geq {\rm\underline{dist}}(R_x^t, \partial C_{x}^{A}),\quad \text{ for any } t>1.
\end{equation}
In addition, due to $R_x\subset K_x$ (see (\ref{W_R_K})) and the definition of $R_y$ (with $y\in\overline{O^+(x)}$),
\begin{equation}\label{R_belong_K}
	R_{\Phi_{t-1}(x)}=\Phi_{t-1}(x)-x+R_x\subset\Phi_{t-1}(x)-x+K_x=K_{\Phi_{t-1}(x)},\quad \text{ for any } t>1.
\end{equation}
Since $\Phi_t$ is strongly competitive, we have
\begin{equation}\label{strongly_comp}
	\Phi_1 K_{\Phi_{t-1}(x)}\backslash\{\Phi_{t}(x)\}\subset \text{Int}K_{\Phi_t(x)},\quad \text{ for any } t>1.
\end{equation}
Together with (\ref{R_belong_K}) and (\ref{strongly_comp}), it yields that	$\Phi_{1}R_{\Phi_{t-1}(x)}\backslash\{\Phi_t(x)\}\subset \text{Int} K_{\Phi_{t}(x)}$, that is, 
\begin{equation}\label{Rxt_belong_K}
	R_x^t\subset \text{Int}K_x\cup\{x\},\quad\text{ for any }t>1.
\end{equation}
Noticing that $\partial C_{x}^{A}\subset\partial C_x\backslash\{x\}$ and $(\text{Int}K_x\cup\{x\})\cap (\partial C_x\backslash\{x\})=\emptyset$, (\ref{Rxt_belong_K}) deduces that the two compact sets $R_x^t$ and $\partial C_{x}^{A}$ are disjoint. This means
\begin{equation*}
	{\rm\underline{dist}}(R_x^t,\partial C_{x}^{A})>0,\quad\text{ for any }t>1.
\end{equation*}
In view of Lemma \ref{4p}(ii), we obtain the continuity of ${\rm\underline{dist}}(x-\Phi_1(y)+\Phi_{1}R_{y}, \partial C_{x}^{A})$ with respect to $y\in\overline{O^{+}(x)}$. 
Then, there exists $\delta>0$ such that
\begin{equation}\label{Sepa_ind>0}
	{\rm\underline{dist}}(R_x^t, \partial C_{x}^{A})\geq\delta,\quad\text{ for any }t>1,
\end{equation}
because $\overline{O^{+}(x)}$ is compact.
Moreover, by (\ref{d>d}) and (\ref{Sepa_ind>0}), 
\begin{equation*}\label{unif_dis_W_C}
	{\rm\underline{dist}}(W_{x}^{t},\partial C_{x}^{A})\geq\delta,\quad \text{ for any } t>1.
\end{equation*}
Thus, we have proved the assertion and (\ref{E:wtx-bdd}).

Now, due to \eqref{E:ome-z}-\eqref{E:wtx-bdd}, we choose an $\eta>0$ small that $$0<\eta<\min\left(\frac{\delta}{2},\,\, {\rm \underline{dist}}(\omega(z),\partial C_{x}\backslash\text{Int}A)\right).$$
Recall that $x\in R(\Phi)$. Then, there exists $t_0>1$ such that $|x-\Phi_{t_{0}}(x)|<\eta$. So,
by the definition of $W_{x}^{t_0}$ and (\ref{dis_1}),
\begin{align}
	d_{H}(\Phi_{t_{0}}W_{x},W_{x}^{t_{0}})&=d_H(\Phi_{t_0}W_{x},x-\Phi_{t_0}(x)+W_{\Phi_{t_0}(x)})\le |x-\Phi_{t_0}(x)|<\eta. \label{E:W-Wt0}
\end{align}
Consequently, together with (\ref{dis_2}), (\ref{E:wtx-bdd}) and (\ref{E:W-Wt0}), we obtain that
\begin{align}
	{\rm \underline{dist}}\left(\Phi_{t_{0}} W_{x}, \partial C_{x}^{A}\right)
	&\geq {\rm \underline{dist}}\left(W_{x}^{t_{0}}, \partial C_{x}^{A}\right)-d_{H}(\Phi_{t_{0}}W_{x},W_{x}^{t_{0}})>\delta-\eta>\eta. \label{E:W-eta-L}
\end{align}
Note also that $\eta<{\rm \underline{dist}}(\omega(z),\partial C_{x}\backslash{\rm Int}A))$ and $\Phi_{t_{0}}W_{x}\subset\omega(z)$ (see \eqref{E:in-W}), we further obtain $${\rm \underline{dist}}(\Phi_{t_0}W_x,\partial C_{x}\backslash{\rm Int}A))>\eta.$$ Together with \eqref{E:W-eta-L}, this entails that
\begin{equation}\label{E:W-boundary}
	{\rm \underline{dist}}(\Phi_{t_{0}}W_{x},\partial C_{x})>\eta,
\end{equation}
which is the first part of the claim.

Finally, we show $\Phi_{t_{0}}W_{x}\subset W_{x}$. To this end, recall that $W_x=K_x\cap\omega(z)$, by the invariance of $\omega(z)$, it suffices to prove that $\Phi_{t_{0}}W_{x}\subset K_x$. Suppose that $\Phi_{t_{0}}W_{x}\not\subset K_x$. 
Then one can find $y\in \Phi_{t_{0}}W_{x}\cap \text{Int}C_{x}$ (see $\text{Int}C_{x}=\mathbb{R}^n\backslash K_x$).
It follows from $y\in \Phi_{t_0}W_x$ and (\ref{E:W-boundary}) that
\begin{equation}\label{y_partial_C}
	{\rm \underline{dist}}(y,\partial C_x)>\eta.
\end{equation}
Noticing that $y\in\text{Int}C_x$, (\ref{y_partial_C}) deduces that $${\rm \underline{dist}}(y,K_x)>\eta.$$
Combined with (\ref{E:wtx}), one has 
\begin{equation}\label{y_Wt0x>eta}
	{\rm \underline{dist}}(y, W_{x}^{t_{0}})>\eta.
\end{equation}
However, noticing that $y\in \Phi_{t_{0}}W_{x}$, it follows from (\ref{E:W-Wt0}) and definition of $d_H$ that $${\rm \underline{dist}}(y, W_{x}^{t_{0}})\leq d_{H}(\Phi_{t_{0}}W_{x},W_{x}^{t_{0}})<\eta.$$
This clearly contradicts (\ref{y_Wt0x>eta}).
Thus, we have proved $\Phi_{t_{0}}W_{x}\subset K_x$, which completes the proof of the claim.
\end{proof}

\begin{prop}\label{sep_n_1}
Let $x\in R(\Phi)$. Then, for any $y\in R(\Phi)$, 
\begin{equation}\label{omega_R}
	\omega(y)\cap(\partial C_x\backslash\{x\})=\emptyset.
\end{equation}
In particular, 
\begin{equation*}\label{u_omega_R}
	\left(\bigcup_{y\in R(\Phi)}\omega(y)\right)\cap \partial C_x=\{x\}.
\end{equation*}
\end{prop}

\begin{proof}
We only need to prove (\ref{omega_R}).
To this end, we first note that if $x\in\omega(y)$, then the non-ordering property of $\omega(y)$ implies (\ref{omega_R}) immediately. Therefore, we hereafter assume that $x\notin\omega(y)$. We will consider the following three cases, respectively.

Case (i): $\omega(y)\subset\partial C_x$. We show this case actually cannot happen. In fact, since $y\in R(\Phi)$, one obtains that $y\in\omega(y)\subset\partial C_x$. But $x\notin\omega(y)$, so $y\in \partial C_x\backslash\{x\}$, which contradicts Proposition \ref{sep_n}.

Case (ii): $\omega(y)\cap\text{Int} C_{x}\neq\emptyset$.
We will show (\ref{omega_R}).
Suppose not, then one can take $y_{0}, y_{1} \in\omega(y)$ such that $x-y_{1}\in\text{Int}C$ and $x-y_{0}\in\partial C$. Since $\Phi_t$ is strongly competitive (see \textnormal{(H1)}), there is a small $\tau>0$ such that $\Phi_{\tau}(x)-\Phi_{\tau}(y_{1})\in\text{Int}C\ \text{and}\ \Phi_{\tau}(x)-\Phi_{\tau}(y_{0})\in\text{Int}K$.
Choose $z_i\in O(y)$ close to $\Phi_{\tau}(y_i)$ for $i=0,1$,
such that $\Phi_{\tau}(x)-z_{1}\in\text{Int}C$ and $\Phi_{\tau}(x)-z_{0}\in\text{Int}K$. Then, the connectedness of $O(y)$ implies that there exists some $z\in O(y)$ such that $z\in \partial C_{\Phi_{\tau}(x)}\backslash\{\Phi_{\tau}(x)\}$. Noticing that $z\in O(y)\subset R(\Phi)$, we obtain a contradiction to Proposition \ref{sep_n}. Therefore, we obtain (\ref{omega_R}).

Case (iii): $\omega(y)\cap\text{Int} K_{x}\neq\emptyset$. Similarly as Case (ii), we take $y_{0}, y_{2} \in\omega(y)$ such that $x-y_{2}\in\text{Int}K$ and $x-y_{0}\in\partial C$. Choose $\tau'<0$ small so that $\Phi_{\tau'}(x)-\Phi_{\tau'}(y_{2})\in\text{Int}K\ \text{and}\ \Phi_{\tau'}(x)-\Phi_{\tau'}(y_{0})\in\text{Int}C$.
By the same argument in Case (ii) as above, one can obtain a contradiction to Proposition \ref{sep_n}.
Thus, we have proved (\ref{omega_R}), which completes the proof.
\end{proof}

Proposition \ref{sep_n_1} helps us obtain a technical lemma, called as ``\textit{limit-set dichotomy for recurrent points}", which turns out to be crucial to obtain the ``solid" intersection principle.
For brevity, for any subsets $A,B\subset \mathbb{R}^n$, we hereafter denote $A\approx B$, which means that $a-b\in\text{Int}C$ for any $a\in A$, $b\in B$.

\begin{lem}\label{limit_set_dichotomy}
{\rm(Limit-Set Dichotomy for Recurrent Points).}
Let $x,y\in R(\Phi)$. Then, either\par
\rm(i) $\omega(x)\approx\omega(y)$; or otherwise,\par
\rm(ii) $\omega(x)\cup\omega(y)$ is unordered.
\end{lem}

\begin{proof}
We first show that, if $\omega(x)\cap\omega(y)\neq\emptyset$, then (ii) holds, that is, $\omega(x)\cup\omega(y)$ is unordered. For otherwise, there exist $x_{1}\in \omega(x)\backslash\omega(y)$ and $y_{1}\in \omega(y)$ such that $x_1-y_1\in\text{Int}C$. Choose $x_2\in O(x)$ close to $x_1$ such that
\begin{equation}\label{x2_y1_C}
	x_2-y_1\in\text{Int}C.
\end{equation}
Let $x_2$ be the $x$ in Proposition \ref{sep_n_1}. Then, Proposition \ref{sep_n_1} implies that $\omega(y)\cap\partial C_{x_2}=\emptyset$ (since $x_2\notin \omega(y)$). Then (\ref{x2_y1_C}) directly entails that $x_2\approx\omega(y)$. In particular, $x_2\approx\omega(x)\cap\omega(y)$.
Thus, we can find an $x_3\in\omega(x)$ such that $x_2- x_3\in\text{Int}C$, which contradicts the non-ordering of $\omega(x)$. Thus, we have proved that $\omega(x)\cup\omega(y)$ is unordered.

Next, we assume that $\omega(x)\cap \omega(y)=\emptyset$. 
We will \textit{assert that there is no points $x_{0}\in\omega(x)$ and $y_{0}\in\omega(y)$ such that $x_{0}-y_{0}\in\partial C$}. 
Before proving this assertion, we show how it implies the conclusion of this lemma.
In fact, choose any $x_1\in\omega(x)$, $y_1\in\omega(y)$. Since $\omega(x)\cap\omega(y)=\emptyset$, one has $x_1\neq y_1$. Moreover, due to \textnormal{(H1)} and the invariance of $\omega(x)$ and $\omega(y)$, we can assume without loss of generality that:
(i) $x_1-y_1\in\text{Int}C$; or otherwise, (ii) $x_1-y_1\in\text{Int}K$.

If (i) holds, then we have
\begin{equation}\label{x1_approx_omegay}
	x_1\approx\omega(y).
\end{equation}
(Otherwise, one can find $w_1\in\omega(y)$ such that $x_1-w_1\notin \text{Int}C$. Note also that $x_1-y_1\in \text{Int}C$. Then it follows from the connectedness of $\omega(y)$ that there is $w_2\in\omega(y)$ such that $x_1-w_2\in\partial C$, contradicting the assertion).
Consequently, (\ref{x1_approx_omegay}) will imply that $x'\approx\omega(y)$ for any $x'\in \omega(x)$.
(Otherwise, there exist $x'\in\omega(x)$ and $y'\in\omega(y)$ such that $x'-y'\notin \text{Int}C$. But, note $x_1-y'\in\text{Int}C$ by (\ref{x1_approx_omegay}), the connectedness of $\omega(x)$ implies that there is a $z\in\omega(x)$ such that $z-y'\in\partial C$, which contradicts the assertion). So, we have obtained that $\omega(x)\approx\omega(y)$ if (i) holds.

If (ii) holds, then similarly as (\ref{x1_approx_omegay}), one can obtain that $\{x_1\}\cup\omega(y)$ is unordered. 
Furthermore, we can show that $\{x'\}\cup\omega(y)$ is unordered for any $x'\in\omega(x)$.
(Otherwise, there exist $x'\in\omega(x)$ and $y'\in\omega(y)$ such that $x'-y'\in C$. Note also $x_1-y'\notin C$. Then, by the connectedness of $\omega(x)$, one can obtain $x_0\in\omega(x)$ such that $x_0-y'\in\partial C$, a contradiction to the assertion).
So, $\omega(x)\cup\omega(y)$ is unordered, if (ii) holds. Therefore, we have proved the lemma provided that the assertion is confirmed.

Finally, it remains to prove the assertion. Suppose that there exist $x_{0}\in\omega(x)$ and $y_{0}\in\omega(y)$ such that $x_{0}-y_{0}\in\partial C$. Then it is clear that $\omega(y)\not\subset\partial C_{x_0}$ (otherwise, $\omega(y)\subset \partial C_{x_0}$, which implies that $y\in\partial C_{x_0}$, or equivalently $x_0\in\partial C_y$,
so, we have $x_0\in\omega(x)\cap(\partial C_y\backslash\{y\})$, contradicting Proposition \ref{sep_n_1}).	
Therefore, we only need to deal with (a): $\omega(y)\cap \text{Int}C_{x_0}\neq\emptyset$; or otherwise, (b): $\omega(y)\cap\text{Int}K_{x_0}\neq\emptyset$.

If Case (a) holds. Then let $y_{1}\in\omega(y)\cap\text{Int} C_{x_{0}}$. Since $\Phi_t$ is strongly competitive, one can find a small $\tau >0$ such that $\Phi_{\tau}(x_{0})-\Phi_{\tau}(y_{1})\in\text{Int}C$ and $\Phi_{\tau}(x_{0})-\Phi_{\tau}(y_{0})\in\text{Int}K$.
Note that $\Phi_{\tau}(y_{0}), \Phi_{\tau}(y_{1})\in\omega(y)$ and $\Phi_{\tau}(x_{0})\in\omega(x)$.
Then one can choose $z_0, z_1\in O(y)$ and $x_{1}\in O(x)$ close to $\Phi_{\tau}(y_{0}),\Phi_{\tau}(y_{1})$ and $\Phi_{\tau}(x_{0})$, respectively, satisfying $x_{1}-z_{1}\in\text{Int}C$ and $x_{1}-z_{0}\in\text{Int}K$.
So, due to the connectedness of $O(y)$, there exists $z\in O(y)$ with $x_1-z\in\partial C$.
Together with $\omega(x)\cap\omega(y)=\emptyset$, it means that $x_{1}\in\partial C_{z}\backslash\{z\}$. This contradicts Proposition \ref{sep_n}, because $x_1, z\in R(\Phi)$.\par

If Case (b) holds, then choose $y_{2}\in\omega(y)\cap \text{Int} K_{x_{0}}$. There is $\tau'<0$ such that $\Phi_{\tau'}(x_{0})-\Phi_{\tau'}(y_{2})\in\text{Int}K$ and $\Phi_{\tau'}(x_{0})-\Phi_{\tau'}(y_{0})\in\text{Int}C$.
Similarly as our arguments in Case (a), we obtain a contradiction again.
Thus, we have proved the assertion and completed the proof.
\end{proof}

Now we are ready to present the following intersection principle.

\begin{prop}\label{endpoint}
{\rm(Intersection Principle for Joint Cone-Boundary).}
Let $x\in R(\Phi)$. Then $\mathcal{B}(\Phi) \cap \partial C_{x} = \{x\}$.
\end{prop}

\begin{proof}
We prove this proposition by contradiction. Suppose that there is $z\in \mathcal{B}(\Phi)\cap (\partial C_{x}\backslash\{x\})$.
Since $\Phi_t$ is strongly competitive by \textnormal{(H1)}, one can find a $\tau>0$ such that $\Phi_{\tau}(x)-\Phi_{\tau}(z)\subset\text{Int}K$ and $\Phi_{-\tau}(x)-\Phi_{-\tau}(z)\subset\text{Int}C$. 

Take a sequence $\{z_n\}_{n\geq1} \subset R(\Phi)$ satisfying $z_n \rightarrow z$ as $n\rightarrow \infty$.
Without loss of generality, we can assume that $\Phi_{\tau}(x)-\Phi_{\tau}(z_{n})\in\text{Int}K$ for all $n\geq1$. 
Noticing $x,z_n\in R(\Phi)$, one has $\Phi_{\tau}(x), \Phi_{\tau}(z_n)\in R(\Phi)$. Then, by Lemma \ref{limit_set_dichotomy} for $\Phi_{\tau}(x)$ and $\Phi_{\tau}(z_n)$, we obtain that $\omega(x)\cup\omega(z_n)$ (which equals to $\omega(\Phi_{\tau}(x))\cup\omega(\Phi_{\tau}(z_n))$) is unordered.
In particular, $x-z_n\in \text{Int}K$ for any $n\geq1$.

On the other hand, we also have $\Phi_{-\tau}(x)-\Phi_{-\tau}(z_n)\subset\text{Int}C$.
Then by Lemma \ref{limit_set_dichotomy} again, we obtain that $\omega(x)\approx\omega(z_n)$ for all $n\geq1$. In particular, $x-z_n\in\text{Int}C$ for all $n\geq1$, a contradiction.
Thus, we have completed the proof.
\end{proof}

Before ending this section, we would like to present a very useful lemma, called ``\textit{connecting lemma}", which is a direct consequence of the intersection principle (Proposition \ref{endpoint}). As we will see in the following sections, it plays an extremely important role for the description of the order structure of the Birkhoff center.

\begin{lem}{\rm(Connecting Lemma).}\label{3.6_2}
Let $S\subset \mathcal{B}(\Phi)$ be any connected set and $x\in R(\Phi)$. If there exist $b,c\in S$ such that $b\in {\rm Int}C_x$ and $c\in{\rm Int}K_x$, then we must have $x\in S$.
\end{lem}
\begin{proof}
By the intersection principle (Proposition \ref{endpoint}), we note that $S\cap (\partial C_{x}\backslash\{x\}) = \emptyset$. 
So, if $x\notin S$, then $S\cap \partial C_{x} = \emptyset$. Define $S_1=S\cap \text{Int}C_x$ and $S_2=S\cap \text{Int}K_x$. Clearly, $S_1,S_2\neq\emptyset$, because $b\in S_1$ and $c\in S_2$. Moreover, $S_1$ and $S_2$ are disjoint relatively open subsets in $S$, and satisfy $S=S_1\cup S_2$. This contradicts to the connectedness of $S$.
\end{proof}

\noindent\section{Geometric Structure of the Birkhoff Center $\mathcal{B}(\Phi)$}\label{Limit dich}

In this section, we will focus on a universal location description for $\mathcal{B}(\Phi)$ and prove the main theorems. Under the standard assumptions \textnormal{(H1)-(H2)}, we will first establish a dichotomy for the order-structure of any connected component of $\mathcal{B}(\Phi)$. More precisely, we have the following

\begin{thm}\label{4.1}
Assume that \textnormal{{(H1)-(H2)}} hold. Let $B$ be a connected component of $\mathcal{B}(\Phi)$. Then one of the following alternatives must occur:\par 
{\rm (i)} $B$ is unordered; or otherwise,\par 
{\rm (ii)} $B$ consists of strongly ordered equilibria.	
\end{thm}
\noindent The proof of Theorem \ref{4.1} will be presented in Subsection \ref{subsec_4.1}.
Moreover, in order to locate the precise positioning of the unordered connected component in Theorem \ref{4.1}(i), we introduce more definitions.

Let $p\in E$ be an equilibrium for $\Phi_t$. A subset $R(p)\subset \mathbb{R}^n$ is called the \textit{basin of repulsion} of $p$ if $$R(p)=\{x\in\mathbb{R}^{n}:\lim_{t\rightarrow\infty}\Phi_{-t}(x)=p\}.$$
We call the subset $$R_{-}(p)=\{x\in R(p): \Phi_{-t}(x)\ll p \text{ for } t>0 \text{ sufficiently large}\}$$ the \textit{basin of lower repulsion} of $p$, and $$R_{+}(p)=\{x\in R(p): p\ll\Phi_{-t}(x) \text{ for } t>0 \text{ sufficiently large}\}$$ the \textit{basin of upper repulsion} of $p$.
The global attractor $\Gamma$ in \textnormal{(H2)} can be characterized as the set of points with bounded orbits, that is, $x\in\mathbb{R}^n\backslash\Gamma$ if and only if $\lim_{t\to\infty}|\Phi_{-t}(x)|=\infty$. Define $$R(\infty)=\{x\in \mathbb{R}^{n}:\lim_{t\to\infty}|\Phi_{-t}(x)|=\infty\},$$ i.e., $R(\infty)=\mathbb{R}^n\backslash\Gamma$.

Fix $v\gg0$, we hereafter denote by $\pm\infty$ the ``elements" which satisfy that $$-\infty\ll tv\ll+\infty,\quad \text{ for any } t\in\mathbb{R}.$$ 

The supremum $\sup{S}$ of a subset $S\subset\mathbb{R}^n$, if it exists, is the minimal point $a$ such that $a\geq S$. The infimum $\inf S$ is defined dually.
We write 
\begin{equation}\label{x_star}
	x_{\star}=\inf \Gamma\ \text{ and }\ \ x^{\star}=\sup \Gamma.
\end{equation}
Define
\begin{equation}\label{R_tilde-}
	{\widetilde{R}}_-(+\infty)=\{x\in R(\infty): x_\star\ll\Phi_{-t}(x) \text{ for } t>0 \text{ sufficiently large}\}
\end{equation}
and $R_-(+\infty)=\text{Int}(\widetilde{R}_-(+\infty))$. Similarly, define 
\begin{equation*}\label{R_tilde+}
	{\widetilde{R}}_+(-\infty)=\{x\in R(\infty): \Phi_{-t}(x)\ll x^\star \text{ for } t>0 \text{ sufficiently large}\}
\end{equation*}
and $R_+(-\infty)=\text{Int}(\widetilde{R}_+(-\infty))$. Clearly, 
\begin{equation}\label{xstar_belong_R-infty}
[[x^{\star},+\infty]]\subset R_-(+\infty)\ \ \text{ and }\ \ [[-\infty,x_{\star}]]\subset R_+(-\infty).
\end{equation}

An open \textit{$(n-1)$-cell} in $\mathbb{R}^n$ is a subset homeomorphic to $\mathbb{R}^{n-1}$.
Given a sequence $\{u_i\}_{i\in\mathbb{N}}\subset\mathbb{R}^n$, we write \textit{$u_i\downarrow z$} if $u_i$ converges to $z$ with $z\ll u_i$, and write \textit{$u_i\uparrow z$} if $u_i$ converges to $z$ with $u_i\ll z$.
A point $z$ is in the \textit{lower boundary $\partial_{-} S$} of a set $S\subset \mathbb{R}^{n}$ provided that there is a sequence $\{u_i\}_{i\in\mathbb{N}}$ in $S$ such that $u_i\downarrow z$, but no sequence $\{v_i\}_{i\in\mathbb{N}}$ in $S$ such that $v_i\uparrow z$. The \textit{upper boundary $\partial_{+} S$} is defined analogously. For $p\in E$, let $$M_{-}(p)=\partial_{-}R_{-}(p)\ \ \text{ and }\ \ M_{+}(p)=\partial_{+}R_{+}(p),$$ 
where $M_{-}(p)$ (resp. $M_+(p)$) is called the \textit{lower} (resp. \textit{upper}) \textit{boundary} of $R_-(p)$ (resp. $R_+(p)$).
Analogously, let $$M_{-}(+\infty)=\partial_{-}R_{-}(+\infty)\ \ \text{ and }\ \  M_{+}(-\infty)=\partial_{+}R_{+}(-\infty).$$ 
More detailed properties of the subsets $M_{\pm}(p)$, $M_{-}(+\infty)$ and $M_{+}(-\infty)$ will be presented in Subsection \ref{subsec_4.2} (c.f. Proposition \ref{prop_of_cells_infinite}).
Roughly speaking, these subsets are unordered invariant open $(n-1)$-cells. 
We point out that the subsets $M_{\pm}(p)$ are essentially due to Hirsch \cite[Theorem 4.2]{H88-3}; while, the subsets $M_{-}(+\infty)$ and $M_{+}(-\infty)$ are newly defined unordered invariant open $(n-1)$-cells.

\begin{thm}\label{4.2}
	Let all the hypotheses in Theorem \ref{4.1} hold. Then for any unordered connected component $B$, there exist $p\in E\cup\{-\infty\}$ and $q\in E\cup\{+\infty\}$ such that 
	\begin{equation*}
		B\subset M_{+}(p)\cap M_{-}(q),
	\end{equation*}
	where $M_{+}(p)$ and $M_{-}(q)$ are invariant, unordered open $(n-1)$-cells.
\end{thm}
\noindent The proof of Theorem \ref{4.2} will be postponed in Subsection \ref{subsec_4.2}.

\vskip 3mm
\noindent\textbf{Remark 4.1.} 
In view of Theorem \ref{4.2}, compared with Hirsch’s construction for Kolmogorov competitive systems (which are only restricted in the first quadrant of $\mathbb{R}^n$), new members have been collected into the new family $\mathcal{F}$ satisfying Theorem \ref{Thm_2.1}. 

\vskip 3mm
To define $\mathcal{F}=\{M_i\}$ of invariant open $(n-1)$-cells in Theorem \ref{Thm_2.1}, let $M_i=M_-(q_i)\neq\emptyset$ for some $q_i\in E\cup\{+\infty\}$. Then, the elements of $\mathcal{F}$ are pairwise disjoint in $\mathbb{R}^n$ (see Propositions \ref{prop_of_cells_infinite}(b) in Subsection \ref{subsec_4.2}). 
Given that these $\{q_i\}$ are mutually isolated,
it can be inferred that the cardinality of $\mathcal{F}$ is at most countable. In addition, Theorem \ref{Thm_2.1} is a combination of Theorem \ref{4.1} and Theorem \ref{4.2}.

\vskip 3mm
\noindent\textbf{Remark 4.2.} 
As a matter of fact, one can alternatively define in Theorem \ref{Thm_2.1} the family $\mathcal{F} = \{M'_j\}$, where each $M'_j=M_{+}(p_j)\neq\emptyset$ for some $p_j\in E\cup\{-\infty\}$. The elements of $\{M'_j\}$ are pairwise disjoint and the cardinality is at most countable.
\vskip 3mm

\subsection{Proof of Theorem \ref{4.1}.}\label{subsec_4.1}
The proof of Theorem \ref{4.1} heavily depends on the intersection principle (Proposition \ref{endpoint}) and the connecting lemma (Lemma \ref{3.6_2}), which have been established in Section \ref{s3}. We give the details as follows:
\begin{proof}[Proof of Theorem \ref{4.1}]
We first show that $B$ satisfies the order-dichotomy, that is, $B$ is either (i) unordered; or otherwise, (ii) totally ordered.

Suppose that $B$ is neither unordered nor totally ordered. Then there are points $w,x,y,z \in B$ such that $w-x\in C$ and $y-z\in K$. 
By \textnormal{(H1)}, one has $\Phi_{-\tau}(w)-\Phi_{-\tau}(x)\in\text{Int}C$ and $\Phi_{\tau}(y)-\Phi_{\tau}(z)\in\text{Int}K$ for some $\tau>0$.
Based on this, we \textit{assert that one can find three distinct points $a,b,c\in B$ satisfying}
\begin{equation}\label{choose_pts}
	a-b\in\text{Int}C\ \ \ and\ \ \ a-c\in\text{Int}K.
\end{equation}
In fact, if $\Phi_{-\tau}(w)-\Phi_{\tau}(y)\in\text{Int}K$, then we choose $a=\Phi_{-\tau}(w)$, $b=\Phi_{-\tau}(x)$ and $c=\Phi_{\tau}(y)$. While, if $\Phi_{-\tau}(w)-\Phi_{\tau}(y)\in C$, we can take $a=\Phi_{\tau-s}(y)$, $b=\Phi_{-\tau-s}(w)$ and $c=\Phi_{\tau-s}(z)$ for small $s\in(0,\tau)$, which all satisfy the assertion. 

Now, by $(\ref{choose_pts})$, choose an $x\in R(\Phi)$ so close to $a$ that $x-b\in \text{Int}C$ and $x-c\in \text{Int}K$. So, we have $B\cap \partial C_{x}\neq \emptyset$, because $B$ is connected. Then it follows from the connecting lemma (Lemma \ref{3.6_2}) that $x\in B$.

Let $B_{C}=B\cap C_{x}$ and $B_{K}=B\cap K_{x}$. Then, $B_{C}$ and $B_{K}$ are clearly closed in $B$, and $B=B_{C}\cup B_{K}$. Moreover, $B_{C}\neq\emptyset$ and $B_{K}\neq\emptyset$, since $b\in B_{C}$, $c\in B_{K}$. 
By virtue of Proposition \ref{endpoint}, we have $B_{C}\subset\text{Int}C_x\cup\{x\}$ and $B_{K}\subset\text{Int}K_x\cup\{x\}$. Hence, $B_{C}\cap B_{K}=\{x\}$, since $B$ is connected.

Furthermore, we will show that $B_{C}$ and $B_{K}$ are both connected. 
We only prove that $B_C$ is connected. The proof of $B_K$ is analogous. Suppose $B_C$ is not connected. Then there are two disjoint nonempty closed subsets $C_{1}$ and $C_{2}$ such that $B_{C}=C_{1}\cup C_{2}$ and $x\in C_{1}$. 
Noticing that $C_{2}\cap B_{K}=\emptyset$,
we have $(C_{1}\cup B_{K})\cap C_{2}=\emptyset$.
Consequently, one can obtain two nonempty disjoint closed subsets $(C_{1}\cup B_{K})$ and $C_{2}$ of $B$ such that $B=(C_{1}\cup B_{K})\cup C_{2}$, which contradicts the connectedness of $B$. Thus, we have proved that $B_C$ is connected.

Now, since $B_{K}$ is connected and $c\in B_{K}$, one can choose some $y\in B_{K}$ close to $x$ such that $y-b\in \text{Int}C$ and $y-x\in \text{Int}K$. Again, choose some $z\in R(\Phi)$ close to $y$ such that $z-b\in \text{Int}C$ and $z-x\in \text{Int}K$. So, by taking $S$, $b$ and $c$ as $B_C$, $b$ and $x$, respectively, in Lemma \ref{3.6_2}, we can obtain that $z\in B_{C}$. This contradicts $z-x\in \text{Int}K$. Thus, we have proved that $B$ satisfies the order-dichotomy.

Finally, we will show that if $B$ is totally ordered, then it must consist of equilibria.
As a matter of fact, since $B$ is totally ordered, it is homeomorphic to a 1-dimensional compact invariant set. 
To see this, let $v\in\mathbb{R}^n$ satisfy $v\gg0$ and ${\rm span}\{v\}=\{\lambda v:\lambda\in\mathbb{R}\}$. Define $Q$ as the projection from $B$ onto ${\rm span}\{v\}$, that is, $Qx = (x\cdot v)v$ for any $x\in B$, where ``$\cdot$" is the scalar product in $\mathbb{R}^n$. Then, $Q: B\to Q(B)\subset {\rm span}\{v\}$ is a one-to-one mapping (otherwise, $B$ contains two distinct points $x_1$ and $x_2$ that satisfy $(x_1-x_2)\cdot v = 0$, which implies that $x_1$, $x_2$ are not related by ``$\leq$", a contradiction to the ordering of $B$). Consequently, $B$ is homeomorphic to a compact connected invariant subset of ${\rm span}\{v\}$.
Moreover, due to the connectedness of $B$, it is homeomorphic to $[0,1]$, we denote the homeomorphism by $h: B\to[0,1]$.
As $B$ is invariant, there exists a flow $\Psi_t$ on $[0,1]$ such that $\Psi_t\circ h=h\circ\Phi_t|_B$. 
Since the recurrent points of $\Psi_t$ in $[0,1]$ are all equilibria, one obtains that any recurrent point of $\Phi_t$ in $B$ is an equilibrium.
Note also that $B$ is a connected component. Then, any point in $B$ can only be approached by the recurrent points in $B$ (which are all equilibria). Therefore, $B$ consists of equilibria. We have completed the proof.
\end{proof}

\subsection{Proof of Theorem \ref{4.2}.}\label{subsec_4.2}
In order to prove Theorem \ref{4.2}, we first give the following proposition to describe the crucial properties of the $(n-1)$-cells mentioned in Theorem \ref{4.2}.

\begin{prop}\label{prop_of_cells_infinite}
	Let $q\in E\cup\{+\infty\}$. Then\par
	{\rm(a)} $R_{-}(q)$ is order-convex  and invariant;\par
	{\rm(b)} $M_{-}(q)$ is an invariant, unordered open $(n-1)$-cell; and moreover, 
	\begin{equation*}
		M_{-}(q)\cap M_{-}(q')=\emptyset,\quad {\rm for\ any\ distinct }\ \ q,q'\in E\cup\{+\infty\};
	\end{equation*}
	
	{\rm(c)} if $x\in\overline{M_{-}(q)}$ and $x\leq q$, then $x\in M_{-}(q)$.\\
	Corresponding results hold for $R_{+}(p)$ and $M_{+}(p)$ for $p\in E\cup\{-\infty\}$.
\end{prop}

\begin{proof}
	For $q\in E$, one can repeat the same arguments in Hirsch \cite[Theorem 4.2]{H88-3} to obtain (a)-(c). We hereafter only focus on the case that $q=+\infty$.
	
	(a). First, we show that $R_{-}(+\infty)$ is order-convex. In fact, take $x,y\in \widetilde{R}_{-}(+\infty)$ with $x\leq y$. Then, there exists $t_0>0$ such that $x_{\star}\leq\Phi_{-t}(x)\leq \Phi_{-t}(z)$ for any $z\in[x,y]$ and $t>t_0$. As a consequence, $\lim_{s\to\infty}\left|\Phi_{-s}(z)\right|=\infty$ for any $z\in[x,y]$. By (\ref{R_tilde-}), it implies that $\widetilde{R}_{-}(+\infty)$ is order-convex.
	Now, take any $x,y\in R_{-}(+\infty)$. Since $R_{-}(+\infty)$ is open, there are $x_1,y_1\subset R_{-}(+\infty)$ such that $[x,y]\subset[[x_1,y_1]]$. Due to the order-convexity of $\widetilde{R}_{-}(+\infty)$, we have $[x_1,y_1]\subset\widetilde{R}_{-}(+\infty)$. So, $[[x_1,y_1]]\subset R_{-}(+\infty)$; and hence, $[x,y]\subset{R}_{-}(+\infty)$. This means that $R_{-}(+\infty)$ is order-convex.
	
	Next, we prove that $R_{-}(+\infty)$ is invariant.
	For this purpose, it suffices to show the invariance of $\widetilde{R}_-(+\infty)$. For any $x\in \widetilde{R}_-(+\infty)$, there is $t_0>0$ such that $x_\star\ll\Phi_{-t}(x)$ for any $t>t_0$. 
	For any $s_0\in\mathbb{R}$, take any $s>0$ such that $s-s_0>t_0$, and then $x_\star\ll\Phi_{-s}(\Phi_{s_0}(x))$. Therefore, by (\ref{R_tilde-}), $\Phi_{s_0}(x)\in \widetilde{R}_-(+\infty)$ for any $s_0\in\mathbb{R}$.
	Hence, $R_{-}(+\infty)$ is invariant.
	
	(b). We first show the invariance of $M_{-}(+\infty)$. For any $z\in M_{-}(+\infty)$, there is a sequence $\{u_i\}_{i\in\mathbb{N}}\subset R_{-}(+\infty)$ such that $u_i\downarrow z$ as $i\to\infty$, but no sequence $\{v_i\}_{i\in\mathbb{N}}\subset R_{-}(+\infty)$ such that $v_i\uparrow z$ as $i\to\infty$.
	
	To show that $\Phi_{-t}(z)\in M_{-}(+\infty)$ for $t>0$, we note that $\Phi_{-t}(u_i)\downarrow\Phi_{-t}(z)$ as $i\to\infty$. So, it suffices to show that there is no sequence $\{y_n\}_{n\in\mathbb{N}}\subset R_{-}(+\infty)$ such that $y_n\uparrow\Phi_{-t}(z)$ as $n\to\infty$. Otherwise, one has $z\in[y_n,\Phi_{-t}(u_i)]$ for some $n$ and $i$. Since $R_{-}(+\infty)$ is order-convex and $y_n,\Phi_{-t}(u_i)\in R_{-}(+\infty)$, it follows that $[y_n,\Phi_{-t}(u_i)]\subset R_{-}(+\infty)$. So, $z\in R_{-}(+\infty)$, contradicting $z\in M_{-}(+\infty)$.
	
	Now we show $\Phi_{t}(z)\in M_{-}(+\infty)$ for $t>0$.
	For each $i\in\mathbb{N}$, we take a neighborhood $U_i$ of $z$ such that $U_i\ll u_i$, and choose $w_i\in\Phi_{t}U_i$ such that $\Phi_{t}(z)\ll w_i$. Since $u_i\downarrow z$, we have $w_i\downarrow \Phi_{t}(z)$ as $i\to\infty$.
	By \textnormal{(H1)}, we have $\left[\left[\Phi_t(z),w_i\right]\right]\subset\Phi_t\left[\left[z,u_i\right]\right]$.
	Due to the invariance and order-convexity of $R_{-}(+\infty)$, one has
	$\Phi_t\left[\left[z,u_i\right]\right]\subset R_-(+\infty)$, which implies $w_i\in R_{-}(+\infty)$ for any $i$. 
	On the other hand, if there exists $\{y_n\}_{n\in\mathbb{N}}\subset R_{-}(+\infty)$ such that $y_n\uparrow \Phi_{t}(z)$ as $n\to\infty$, then $\Phi_{-t}(y_n)\uparrow z$ as $n\to\infty$, contradicting $z\in M_{-}(+\infty)$.
	Thus, we have proved the invariance of $M_{-}(+\infty)$.
	
	Suppose that $M_{-}(+\infty)$ is not unordered.
	Then, we take $x<y$ in $M_{-}(+\infty)$. 
	By \textnormal{(H1)}, we may assume that $x\ll y$, because $M_{-}(+\infty)$ is invariant.
	Let $\{u_i\}_{i\in\mathbb{N}}$, $\{v_j\}_{j\in\mathbb{N}}$ be two sequences in $R_{-}(+\infty)$ satisfying $u_i\downarrow x$ and $v_j\downarrow y$ as $i,j\to\infty$, respectively.
	Then, $u_k\ll y\ll v_k$ for all $k$ sufficiently large. Therefore, $y\in R_{-}(+\infty)$, since $R_{-}(+\infty)$ is order-convex. This contradicts $y\in M_{-}(+\infty)$.
	Hence, $M_{-}(+\infty)$ is unordered.
	
	We now prove that $M_{-}(+\infty)$ is an open $(n-1)$-cell.
	To see this, let $\mathbb{H}=\{y+\lambda v: y\in H,\lambda\geq0\}$ be a closed half-space of $\mathbb{R}^n$, where $H=\{y\in\mathbb{R}^n: y\cdot v=0\}$ denotes the orthogonal hyperplane of $v\gg0$ and ``$\cdot$" is the  scalar product in $\mathbb{R}^n$.
	For each $y\in H$, we define $$\mu_y=\inf{\{\mu\in\mathbb{R}: y+\mu v\in R_{-}(+\infty)\}}.$$
	It has $\mu_y\leq\inf\{\mu\in\mathbb{R}: y+\mu v\geq x^{\star}\}$, since $[[x^\star,+\infty]]\subset R_{-}(+\infty)$ for any $y\in H$ (see (\ref{xstar_belong_R-infty})). So, $\mu_y<\infty$ for any $y\in H$. 
	Furthermore, one can obtain a homeomorphism 
	\begin{equation}\label{homo_h}
		h:\ \mathbb{H}\longrightarrow \overline{R_{-}(+\infty)},\quad y+\lambda v\mapsto h(y+\lambda v)\triangleq y+ \lambda v + \mu_y v.
	\end{equation}
	Then it follows that $M_{-}(+\infty)$ is an open $(n-1)$-cell, since $M_{-}(+\infty)=h(H)$. 
	
	Finally, we prove that $M_{-}(+\infty)\cap M_{-}(q)=\emptyset$ for $q\in E$. Suppose that $z\in M_{-}(+\infty)\cap M_{-}(q)$, there exist $\{u_i\}_{i\in\mathbb{N}}\subset R_{-}(+\infty)$ and $\{v_j\}_{j\in\mathbb{N}}\subset R_{-}(q)$ such that $u_i\downarrow z$ and $v_j\downarrow z$ as $i,j\to\infty$, respectively. Choose $i_1,i_2,j_0\in\mathbb{N}$ such that $v_{j_0}\in [[u_{i_1},u_{i_2}]]\subset R_{-}(+\infty)$, which implies that $v_{j_0}\in R_{-}(q)\cap R_{-}(+\infty)\neq\emptyset$.  
	Noticing that $R_{-}(q)\subset \Gamma$ and $R_{-}(+\infty)\subset R(\infty)=\mathbb{R}^n\backslash\Gamma$, we obtain a contradiction.
	
	(c). Recall that $M_{-}(+\infty)=h(H)$ in the above paragraph. Since $H$ is closed and $h$ is a homeomorphism (see (\ref{homo_h})), $M_{-}(+\infty)$ is hence closed as well.
	
	Thus, we have completed the proof.
\end{proof}

Before we give the proof of Theorem \ref{4.2}, we still need the following lemma, called ``\textit{absorbing principle for Birkhoff center}". For brevity, we hereafter always let $B$ be an unordered connected component of $\mathcal{B}(\Phi)$.

\begin{lem}\label{k_2}{\rm (Absorbing Principle for Birkhoff center).}
If $q\in E$ satisfies $x<q$ for some $x\in B$, then $B\ll q$. Similarly, if $q\in E$ satisfies $q<x$ for some $x\in B$, then $q\ll B$.
\end{lem}
\begin{proof}
	We only prove the first part of the lemma.
	Suppose on the contrary that there is some $y \in B$ such that $y-q\in K$.
	Then by \textnormal{(H1)} and the invariance of $B$, one may assume that $y-q\in\text{Int}K$ and $x- q\in\text{Int}C$. So, the connecting lemma (Lemma \ref{3.6_2}) entails that $q\in B$ at once. 
	In other words, $B$ contains two points $x<q$, which contradicts the non-ordering of $B$. Thus, we have proved $B\ll q$.
\end{proof}

\vskip 2mm
\noindent\textbf{Remark 4.3.} 
When $B$ is the $\omega$-limit set of an individual orbit, Lemma \ref{k_2} was first proved by Hirsch \cite[Lemma 3.3]{H88-3}. This lemma plays a critical role for proving that the limit set lies on one element of $\mathcal{F}$ in \cite{H88-3}. In our situation, however, $B$ is just a connected component of $\mathcal{B}(\Phi)$, whose structure cannot be fully understood only by means of the analysis of the limit-sets of individual orbits. Fortunately, the intersection principle for joint cone-boundary (Proposition \ref{endpoint}) and connecting lemma (Lemma \ref{3.6_2}) help us again overcome the difficulty and ensure the feasibility of Lemma \ref{k_2} for the absorbing principle for Birkhoff center.
\vskip 2mm

\begin{lem}\label{k_1}
	Let $x\in B$, then one of the following alternatives must occur:\par
	{\rm(i)} there exist a sequence $z_n\downarrow x$, and $q_n\in\alpha(z_n)\cap E$ with $x\leq q_n$; or otherwise,\par
	{\rm(ii)} there exists a sequence $\{z_n\}_{n\in\mathbb{N}}\subset R_{-}(+\infty)$ satisfying $z_n\downarrow x$.
\end{lem}

\begin{proof}	
We consider the following two cases one by one.

Case (i). There exists a neighborhood $U$ of $x$ such that $U\cap[[x,+\infty]]\subset\Gamma$. 
Then, by Generic Convergence Theorem for $\Phi_{-t}$ ($t\geq0$) (see \cite[Theorem 1.4.3]{H95}), there is a sequence $\{z_n\}_{n\in\mathbb{N}}\subset U$ such that $z_n\downarrow x$ and $\alpha(z_n)\subset E$ for all $n\geq1$.
Since $x\in B\subset \mathcal{B}(\Phi)$, $x$ is non-wandering. Hence, there exist two sequences $x_i\to x$ and $t_i\to\infty$ such that $\Phi_{t_i}(x_i)\to x$ as $i\rightarrow \infty$.
Consequently, for each $n\geq1$, $\Phi_{t_i}(x_i)\ll z_n$ for $i$ large. By (H1), we then obtain 
\begin{equation}\label{x<Phix}
	x_i\ll\Phi_{-t_i}(z_n),
\end{equation}
As $\alpha(z_n)\subset E$, we can find a subsequence of $t_i$ (still denoted by $t_i$), such that $\Phi_{-t_{i}}(z_n)\rightarrow q_{n} \in E$ as $i\rightarrow\infty$, for each $n\geq1$. So, by letting $i\to\infty$ in (\ref{x<Phix}), we obtain $x\leq q_{n}$ for all $n\geq1$.

Case (ii). For any neighborhood $V$ of $x$, there is some $z\in V$ with $x\ll z$ such that $z\notin\Gamma$.
In this case, by the invariance of $B$, one has $x_\star\leq\Phi_{-t}(x)\ll\Phi_{-t}(z)$ for any $t>0$, where $x_{\star}=\inf \Gamma$ defined in (\ref{x_star}). Thus, by (\ref{R_tilde-}), we deduce that $z\in \widetilde{R}_{-}(+\infty)$. Further, $z\in R_{-}(+\infty)$.
Since $V$ is arbitrary chosen, we obtain a sequence $\{z_n\}_{n\in\mathbb{N}}\subset R_{-}(+\infty)$ such that $z_n\downarrow x$.
This completes the proof.
\end{proof}

\vskip 2mm
\noindent\textbf{Remark 4.4.} 
Similarly, for $x\in B$, one of the following corresponding alternatives also occurs:\par
{\rm(i)} there exist a sequence $z_n\uparrow x$, and $q_n\in\alpha(z_n)\cap E$ with $q_n\leq x$; or otherwise,\par
{\rm(ii)} there exists a sequence $\{z_n\}_{n\in\mathbb{N}}\subset R_{+}(-\infty)$ satisfying $z_n\uparrow x$.\\
The corresponding result can be proved similarly as in Lemma \ref{k_1}.
\vskip 2mm

By virtue of Lemma \ref{k_1}, we define two subsets of $B$ as
\begin{equation*}
	B_1=\{x\in B: x\text{ satisfies Lemma \ref{k_1}(i) with } x<q_n \text{ for all }n\geq1\},
\end{equation*}
and
\begin{equation*}
	B_2=\{x\in B: x\text{ satisfies Lemma \ref{k_1}(ii)}\}.
\end{equation*}

\begin{lem}\label{B1B2}
{\rm(i)} If $B_i\neq\emptyset$, then $B_j=\emptyset$ for $j\neq i$;\par 
{\rm(ii)} If $B_i\neq\emptyset$, then $B_i$ is dense in $B$, for $i=1,2$.
\end{lem}
\begin{proof}
We first \textit{claim that $B_1\cup  B_2$ is dense in $B$}.
Otherwise, there is a relatively open subset $U$ in $B$ such that $U\subset B\backslash(B_1\cup B_2)$. 
Then, for any $x\in U$, by Lemma \ref{k_1}(i), there exist $z\gg x$ and $q\in\alpha(z)\cap E$ such that
\begin{equation}\label{x=qzx}
	x=q\in \alpha(z)\cap E.
\end{equation}
Hence, $U\subset E$. 
In particular, together with the Monotone Criterion (see more details in \cite[Theorem 1.2.1]{H95}), (\ref{x=qzx}) directly yields 
\begin{equation}\label{zx_x}
	\alpha(z)=\{x\}.
\end{equation}
Take a small neighborhood $N$ of $x$ in $U$ ($\subset E$) such that $N\ll z$ and let $b=\sup{N}$. Clearly, $N<b\leq z$, since $N$ is unordered. 
Moreover, by \textnormal{(H1)} and the invariance of $N$ ($N$ consists of equilibria), $N<\Phi_{-t}(b)$ for any $t>0$. This implies that
\begin{equation}\label{b<Phib<zx}
	x<b\leq\Phi_{-t}(b)\leq\Phi_{-t}(z),\quad \text{ for any } t>0.
\end{equation}
This contradicts (\ref{zx_x}). So, we have proved the claim.

Next, we show that at least one of $B_1$ and $B_2$ is empty.
To explain this, we observe that if $B_1\neq\emptyset$ then $B_2$ must be empty. In fact, if $B_1\neq\emptyset$, then there exists $x\in B$ that satisfies Lemma \ref{k_1}(i) and for which $x<q_{n}\in E$ for some $n\in\mathbb{N}$. Then, Lemma \ref{k_2} guarantees that $B\ll q_{n}$. 
Consequently, $B_{2}=\emptyset$.
Hence, $B_2\neq\emptyset$ also implies $B_1=\emptyset$.
Together with the claim, we have completed the proof.
\end{proof}

Now we are ready to prove Theorem \ref{4.2}.
\begin{proof}[Proof of Theorem \ref{4.2}]
We only prove that $B\subset M_{-}(q)$ for some $q\in E\cup\{+\infty\}$. The proof of $B\subset M_{+}(p)$ for some $p\in E\cup\{-\infty\}$ is similar.
To this end, we will deal with the following two cases, respectively.

Case (a): $B_1\neq \emptyset$ (hence $B_2=\emptyset$ by Lemma \ref{B1B2}(i)).
Let $$E^B=\{e\in E:B\ll e\}.$$
By Lemma \ref{k_2} and $B_1\neq \emptyset$, we have $\emptyset\neq E^B\subset E$. 
Let $q_0=\sup B$. Since $B$ is unordered, $B<q_0$. It then follows from \textnormal{(H1)} and the invariance of $B$ that $q_0< \Phi_{-t}(q_0)$ for any $t>0$. Again, by virtue of the Monotone Criterion, we have
\begin{equation}\label{q_0_to_q}
	\Phi_{-t}(q_0)\to q\in E,\quad \text{ as }\ t\to\infty.
\end{equation}
Therefore, $B<q_0\ll q$. Hence, $q\in E^B$.
In addition, recall that $q_0\leq E^B$. Then $\Phi_{-t}(q_0)\leq E^B$ for $t>0$, which implies that $q\leq E^B$. Therefore, $q=\inf E^B$.

Fix any $x\in B_1$. Since $B\ll q$, there exist a sequence $\{z_n\}_{n\in\mathbb{N}}\ll q$ in Lemma \ref{k_1}(i) such that $q_{n}\leq q$ for any $n\in\mathbb{N}$. 
On the other hand, note that $q_n\in E^B$. Then $q\ (=\inf E^B)\leq q_n$ for all $n\geq1$. This implies that
\begin{equation}\label{q_n=q}
	q_{n}=q,\quad \text{ for any } x\in B_1 \text{ and } n\geq1.
\end{equation}

Now, by Lemma \ref{k_1}(i) and (\ref{q_n=q}), for any $x\in B_1$ with $\{z_n\}_{n\in\mathbb{N}}$, $\{q_n\}_{n\in\mathbb{N}}$ satisfying $z_n\downarrow x$ and $q=q_n\in\alpha(z_n)$ for each $n\in\mathbb{N}$. Again, by the Convergence Criterion, $\alpha(z_n)=\{q\}$. Therefore, $z_n\in R_{-}(q)$ for any $n\in\mathbb{N}$. 
In other words, there is a sequence $\{z_n\}_{n\in\mathbb{N}}\subset R_-(q)$ satisfying that $z_n\downarrow x$.
Together with the fact that $R_{-}(q)$ is order-convex (see Proposition \ref{prop_of_cells_infinite}(a)), there exists no sequence $\{y_i\}_{i\in\mathbb{N}}$ with $y_i\uparrow x$ (otherwise, one has $x\in [y_i,z_n]\subset R_{-}(q)$ for some $i$ and $n$, which contradicts $x\notin R_{-}(q)$).
Thus, we have $x\in M_{-}(q)$. Due to the arbitrariness of $x\in B_1$, we obtain $B_1\subset M_{-}(q)$. Since $B_1$ is dense in $B$ (see Lemma \ref{B1B2}(ii)), it follows from Proposition \ref{prop_of_cells_infinite}(c) that $B\subset M_{-}(q)$.

Case (b): $B_2\neq\emptyset$ (hence $B_1=\emptyset$ by Lemma \ref{B1B2}(i)).
Clearly, for any $x\in B_2$, there exists a sequence $\{z_n\}_{n\in\mathbb{N}}\subset R_{-}(+\infty)$ such that $z_n\downarrow x$ as $n\to\infty$. 
Together with the fact that $R_{-}(+\infty)$ is order-convex in Proposition \ref{prop_of_cells_infinite}(a), there exists no sequence $\{y_i\}_{i\in\mathbb{N}}\subset R_{-}(+\infty)$ with $y_i\uparrow x$ (otherwise, $x\in [y_i,z_n]\subset R_{-}(+\infty)$ for some $i$ and $n$, which contradicts $x\notin R_{-}(+\infty)$).
Thus, we have $x\subset M_{-}(+\infty)$. Due to the arbitrariness of $x\in B_2$, we have $B_2\subset M_{-}(+\infty)$. Again, since $B_2$ is dense in $B$ (see Lemma \ref{B1B2}(ii)), it follows from Proposition \ref{prop_of_cells_infinite}(c) that $B\subset M_{-}(+\infty)$.
We have completed the proof.
\end{proof}

\appendix
\section*{Appendix}

\setcounter{equation}{0} \setcounter{thm}{0}
\renewcommand{\thesection}{\Alph{section}}
\renewcommand{\theequation}{\Alph{section}.\arabic{equation}}
\setcounter{section}{1}
\renewcommand{\thethm}{\Alph{thm}}

In this appendix, we will exhibit some critical properties for recurrent-time and IP sets.
Let $z\in R(\Phi)$ and $\theta>0$, the \textit{recurrent-time set} of $z$ parameterized by $\theta$ is defined as $$N(z,\theta)=\{t>0: |\Phi_{t}(z)-z|<\theta\}.$$

\newtheorem*{proposition}{\textnormal{\textbf{Proposition A.1}}}
\begin{proposition}\label{1p}
Let $z\in R(\Phi)$ and $\theta>0$. Then, for any $\tau,\varepsilon>0$,
\begin{equation}\label{1pp}
	N(z,\theta)\cap\mathcal{T}(\tau,\varepsilon)\neq\emptyset,\tag{A1}
\end{equation}
where $\mathcal{T}(\tau,\varepsilon)=\{n\tau+t:n\in\mathbb{N},|t|<\varepsilon\}$.
\end{proposition}

\begin{proof}
We first claim that there is a time sequence $\{p_{i}\}_{i\geq1}$ such that 
\begin{equation}\label{claim}
	S\triangleq \{ p_{i_{1}}+\cdots+p_{i_{k}}:i_{1}<\cdots<i_{k},k\geq1\}\subset N(z,\theta).\tag{A2}
\end{equation}
Actually, the construction of $S$, motivated by Furstenberg \cite[Proposition 8.10]{F81} (see also Ye et al. \cite[Theorem 1.2.13]{Y08}), is called IP-set for the discrete-time cases. For the completeness, we give the details of the construction of $S$.
In fact, since $z$ is recurrent, there exists $p_{1} > 0$ such that
\begin{equation}\label{1}
	|\Phi_{p_{1}}(z)-z|<\theta.\tag{A3}
\end{equation}
Then one can choose some $\theta_{1}\in (0,\theta]$ such that 
\begin{equation}\label{2}
	|\Phi_{p_{1}}(y)-z|< \theta,\quad\text{ whenever }|y-z|< \theta_{1}.\tag{A4}
\end{equation}
For such $\theta_{1}\in (0,\theta]$, we take $p_{2}>0$ such that 
\begin{equation}\label{31}
	|\Phi_{p_{2}}(z)-z|<\theta_{1}.\tag{A5}
\end{equation}
So, by \eqref{1}-\eqref{31}, we have 
\begin{equation*}
	|\Phi_{s}(z)-z|<\theta,\quad \text{ for }s=p_{1},p_{2}\text{ and }p_{1}+p_{2}.
\end{equation*}
Suppose now that $p_{1},\cdots,p_{n}$ have been determined such that 
\begin{equation}\label{5}
	|\Phi_{s}(z)-z|<\theta,\quad \text{ for } s=p_{i_{1}}+\cdots+p_{i_{k}}\text{ and }1\leq i_{1}< \cdots < i_{k}\leq n.\tag{A6}
\end{equation}
We show how to choose $p_{n+1}$. Similarly as above, from~(\ref{5}), we choose $\theta_{n+1}\in(0,\theta]$ such that 
\begin{equation}\label{6}
	|\Phi_{s}(y)-z|<\theta,\quad\text{ whenever }|y-z|< \theta_{n+1},\text{ and }s \text{ in }({\rm\ref{5}}).\tag{A7}
\end{equation}
For such $\theta_{n+1}$, we take $p_{n+1}>0$ such that
\begin{equation}\label{7}
	|\Phi_{p_{n+1}}(z)-z|<\theta_{n+1}.\tag{A8}
\end{equation}
Then, by \eqref{5}-\eqref{7}, we obtain that
\begin{equation*}
	|\Phi_{s}(z)-z|<\theta,\quad \text{ for } s=p_{i_{1}}+\cdots+p_{i_{k}}\text{ and }1\leq i_{1}< \cdots < i_{k}\leq n+1.
\end{equation*}
By following these steps, we can obtain the sequence $\{p_{i}\}_{i\geq1}$ satisfying~(\ref{claim}), which completes the construction of $S$ and the claim. 

Now, choose an integer $n>0$ such that $\dfrac{\tau}{n}<\varepsilon$. For such $n>0$, define 
\begin{equation*}  N_k=
	\begin{cases}
		[\frac{1}{n-k+1},\frac{1}{n-k}], & \text{if }k=1,\cdots,n-1; \\
		[-\frac{1}{n},\frac{1}{n}], & \text{if }k=0; \\
		[-\frac{1}{n+k},-\frac{1}{n+k+1}], & \text{if }k=-1,\cdots,-(n-1).
	\end{cases}
\end{equation*}
Then $S$ will possess a partition as $$S=\bigcup_{k=-(n-1)}^{n-1}Q_{k},$$
where 
\begin{equation*}
	Q_{k}= S\cap \{m\tau+\tau\cdot N_{k}: m\in\mathbb{N}\}.
\end{equation*}
Clearly, $q\in Q_{k}$ if and only if 
\begin{equation}\label{Q}
	q-m\tau\in \tau\cdot N_{k},\quad \text{ for some integer }m\geq0. \tag{A9}
\end{equation}
We observe that in order to prove~(\ref{1pp}), it suffices to prove $Q_0\neq \emptyset$ (Indeed, if $p\in Q_{0}\neq \emptyset$, then there exists an integer $n_{0}\geq0$ with $\left| p-n_{0}\tau\right| < \varepsilon$. So, $p\in S\cap \mathcal{T}(\tau,\varepsilon)\subset N(z,\theta)\cap \mathcal{T}(\tau,\varepsilon)\neq\emptyset$).

For this purpose, choose an integer $k_{1}\in [-(n-1),n-1]$ and a subsequence $\{p^{1}_{i}\}_{i \geq 1}$ of $\{p_{i}\}_{i \geq 1}$ such that $\{p^{1}_{i}\}_{i \geq 1}\subset Q_{k_{1}}$. If $k_1=0$, we've done. In the following, we only focus on the case of $k_{1}>0$. The case of $k_{1}<0$ is analogous.

Note that $N_{k_{1}}=[\frac{1}{n-k_1+1},\frac{1}{n-k_1}]$.
By (\ref{Q}), for any $p^{1}_{j}\in\{p^{1}_{i}\}_{i\geq1}$, there exists an integer $m^{1}_{j}\geq0$ such that $p^{1}_{j}-m_{j}^1 \tau\in \tau\cdot N_{k_{1}}$.
For any $i\geq1$, we set $$p^{2}_{i}=\sum_{j=(i-1)(n-k_1)+1}^{i(n-k_1)}p^{1}_{j}\ \ \text{ and }\ \ m_i^2=1+\sum_{j=(i-1)(n-k_1)+1}^{i(n-k_1)}a_{j}^1.$$
Clearly, $p_i^2\in S$, and
\begin{equation*}
	\begin{split}
		p_i^2 - m_i^2\tau
		&= \sum_{j=(i-1)(n-k_1)+1}^{i(n-k_1)}(p^{1}_{j}-m_{j}^1\tau)-\tau   \\
		&\in(n-k_1)\tau\cdot\left[\frac{1}{n-k_1+1},\frac{1}{n-k_1}\right]-\tau \\
		&= \tau\cdot \left[-\frac{1}{n-k_1+1},0\right]
		\subset \bigcup_{j=-(k_{1}-1)}^{0} \tau\cdot N_{j}.
	\end{split}
\end{equation*}
By (\ref{Q}), it entails that $\{p^{2}_{i}\}_{i\geq1}\subset \bigcup_{-(k_1-1)\leq j \leq 0} Q_{j}$.

Since $\{p_{i}^2\}_{i\geq1}$ is countable, there is an integer $k_{2}$ and a subsequence of $\{p_{i}^2\}_{i\geq 1}$, still denoted by $\{p^{2}_{i}\}_{i\geq1}$, such that $k_2\in (-k_1, 0]$ and $\{p^{2}_{i}\}_{i\geq 1}\subset Q_{k_{2}}$. If $k_2=0$, we've done again. If $k_{2}<0$, then $N_{k_{2}}=[-\frac{1}{n+k_2},-\frac{1}{n+k_2+1}]$ and $p_i^2 - m_i^2\tau \subset \tau\cdot N_{k_2}$ for any $i\geq1$.

For $i\geq1$, we write $$p^{3}_{i}=\sum_{j=(i-1)(n+k_2)+1}^{i(n+k_2)}p^{2}_{j}\ \ \text{ and }\ \  m_i^3=-1+\sum_{j=(i-1)(n+k_2)+1}^{i(n+k_2)}m_{j}^2.$$
So, $p_i^3\in S$, and
\begin{equation*}
\begin{split}
	p_i^3 - m_i^3\tau
	&= \sum_{j=(i-1)(n+k_2)+1}^{i(n+k_2)}(p^{2}_{j}-m_{j}^2\tau)+\tau   \\
	&\in-(n+k_2)\tau\cdot\left[\frac{1}{n+k_2},\frac{1}{n+k_2+1}\right]+\tau \\
	&= \tau\cdot \left[0, \frac{1}{n+k_2+1}\right]
	\subset \bigcup_{j=0}^{-k_2-1} \tau\cdot N_{j}.
\end{split}
\end{equation*}
Again, by (\ref{Q}), we obtain $\{p_i^3\}_{i\geq1}\subset \bigcup_{0\leq j\leq -k_2-1} Q_{j}$.

Similarly as above, choose an integer $k_{3}\geq0$ and a subsequence of $\{p_{i}^3\}$, still denoted by $\{p_{i}^3\}$, such that $k_3\in [0,-k_2)$ and $\{p^{3}_{i}\}\subset Q_{k_{3}}$. If $k_3=0$, we've done. If $k_{3}>0$, we can repeat the same argument for $k_{1}$.

By repeating the steps above, we finally obtain a strictly decreasing integer sequence $|k_1|>|k_2|>\cdots\geq0$. Therefore, one can obtain $k_{l}=0$ for some $l$. In other words, we obtain $\{k_{i}\}_{i=1}^{l}$ satisfying $|k_1|>|k_2|>\cdots>|k_l|=0$, such that each $Q_{k_{i}}$ is at least countable for any $i=1,\dots,l$. In particular, we have proved the fact that $Q_{k_{l}}(=Q_0)$ is nonempty.
\end{proof}

\end{document}